\documentclass{article}[11pt]

\usepackage{algorithm}
\usepackage{float}
\usepackage{algpseudocode}
\usepackage{fullpage}
\usepackage{color}

\usepackage{amssymb}
\usepackage{amsmath,amsthm,amssymb,graphicx}
\usepackage[latin1]{inputenc} 
\usepackage[figuresright]{rotating}

\newtheorem{theorem}{Theorem}
\newtheorem{remark}[theorem]{Remark}
\newtheorem{corollary}[theorem]{Corollary}

\newtheorem{lemma}[theorem]{Lemma}

\newcommand{\Z}{{\mathbb Z}}
\newcommand{\gr}{\gamma_{gr}}
\newcommand{\R}{\mathcal R}
\newcommand{\GR}{G\hookleftarrow \R}

\begin{document}

\title{Grundy dominating sequences on $X$-join product\thanks{\small Partially supported by grants PICT-2016-0410 (MINCyT 2017-2019) and PID-UNR ING 538 (2017-2020).}}

\author{
Graciela Nasini
\and
Pablo Torres
}

\date{}

\maketitle

\begin{center}
Facultad de Ciencias Exactas, Ingenier\'ia y Agrimensura, Universidad Nacional de Rosario

and\\
Consejo Nacional de Investigaciones Cient\' ificas y T\' ecnicas, 
Argentina.\\
\{\texttt{nasini,ptorres}\}\texttt{@fceia.unr.edu.ar} \\
\end{center}

\begin{abstract}
{\noindent
In this paper we study the Grundy domination number on the $X$-join product $\GR$ of a graph $G$ and a family of graphs $\R=\{G_v: v\in V(G)\}$.  The results led us to extend the few known families of graphs where this parameter can be efficiently computed. We prove that if, for all $v\in V(G)$, the Grundy domination number of $G_v$ is given, and $G$ is a power of a cycle, a power of a path, or a split graph, computing the Grundy domination number of $\GR$ can be done in polynomial time. In particular, the results for power of cycles and paths are derived from a polynomial reduction to the Maximum Weight Independent Set problem on these graphs. 

As a consequence, we derive closed formulas to compute the Grundy domination number of the lexicographic product $G\circ H$ when $G$ is a power of a cycle, a power of a path or a split graph, generalizing the results on cycles and paths given by Bre\v sar et al. in 2016. Moreover, the results on the $X$-join product when $G$ is a split graph also provide polynomial-time algorithms to compute the Grundy domination number for $(q,q-4)$ graphs, partner limited graphs and extended $P_4$-laden graphs,  graph classes which are high in the hierarchy of few $P_4$'s graphs.
}

\end{abstract}

\noindent
{\bf Keywords:}  Grundy dominating sequences, $X$-join product, split graphs, power of paths, power of cycles\\
\section{Introduction}

In a similar fashion as Grundy number of graphs relate to greedy colorings, Grundy \emph{domination} number refers to greedy dominating sets. A \emph{greedy domination procedure} applied to a graph generates a sequence of vertices  such that 
each vertex in the sequence has a \emph{private neighbour} which has not been dominated by the previous ones and every vertex in the graph is dominated by at least one vertex in the sequence. We say that such a sequence of vertices is  \emph{legal dominating}. While the length of a shortest legal dominating sequence is the domination number of $G$, the length of a longest one provides an upper bound for the size of dominating sets that can be constructed by a greedy domination procedure. Regarding algebraic properties, a strong connection between the Grundy domination number and the zero forcing
number of a graph \cite{Zero} was established in \cite{BresarZero}. 

A \emph{Grundy dominating sequence} (\textit{Gds} for short) is a legal dominating sequence of maximum length and its length, denoted $\gr(G)$, is the \emph{Grundy domination number of} $G$ (technical definitions are given in the next section). 
In \cite{martin} it is proved that obtaining the Grundy domination number is 
$NP$-hard, even for chordal graphs. On the other hand, efficient algorithms for trees, cographs and split graphs have been presented.

As it is known, every branch of mathematics employs some notion of a product that enables the combination
or decomposition of its elemental structures. In graph theory, cartesian product, strong product, direct product and lexicographic product are the four main products, each with its own set of applications and theoretical interpretations (see, e.g. \cite{book}). Moreover, knowing the relationship among graph parameters of the product and those of its factors allows 
the study of these parameters in graph families  
and usually derive in the design of efficient algorithms to compute them (see, e.g. \cite{FewP4s}).

In particular, the known modular decomposition of graphs \cite{modular1} can be seen as obtaining the prime factors of a graph with respect to the $X$-\emph{join product}, an important generalization of the lexicographic product, introduced by Hiragushi in 1951 \cite{Hir51}. This product is also known as \emph{joint of family of graphs} \cite{Jol73}, \emph{substitution decomposition} \cite{MR84} and \emph{lexicographic sum} \cite{Sabidussi}. In this paper, we use the name $X$-join and we adopt the following notation. Given a graph $G$ and a graph family $\mathcal R=\{G_v: v\in V(G)\}$, the $X$-join product of $G$ and $\mathcal R$ is the graph denoted by $\GR$ and obtained by replacing every vertex $v$ of $G$ by the graph $G_v$.
If every $G_v$ is isomorphic to a graph $H$ we obtain the lexicographic product of $G$ and $H$, denoted by $G\circ H$.
Moreover, for graphs $G$ of order two and $\mathcal R=\{G_1,G_2\}$, if $G$ is complete (resp. edgeless) then $\GR$ is the \emph{complete join} (resp. disjoin union) of $G_1$ and $G_2$, usually denoted by $G_1\vee G_2$ (resp. $G_1+G_2$).

In \cite{Bresar1} the Grundy domination number of grid-like, cylindrical and toroidal graphs was
studied. More precisely, the four standard graph products of paths and/or cycles were considered, and exact formulas for the Grundy domination numbers were obtained for most of the products with two path/cycle factors. 
Concerning the lexicographic product they provided a formula for $\gr (G \circ H)$ which depends on $\gr(H)$ and a particular parameter over dominating sequences of $G$. In particular, closed formulas in terms of $\gr(H)$ are obtained when $G$ is a path or a cycle.




In this paper we study the parameter Grundy domination number for the $X$-join product of graphs, generalizing the results in \cite{Bresar1} for the lexicographic product.
We prove that, if $G$ is a power of a cycle or a power of a path, given $\gr(G_v)$ for all $v\in V(G)$, computing $\gr(G\hookleftarrow \R)$ can be reduced (in polynomial time) to the Maximum Weight Independent Set problem on $G$, which can be solved in polynomial time (see, e.g., \cite{mwis}). 
We also prove that $\gr(G\hookleftarrow \R)$ is polynomial-time computable when $G$ is a split graph, with similar consequence for the lexicographic product.
These results 
also provide polynomial-time algorithms to compute the Grundy domination number for several graph classes that are high in the hierarchy of the classes of graphs recognizable by modular decomposition. More specifically, we work with $(q,q-4)$ graphs \cite{qq-4}, partner limited graphs ($PL$) \cite{Roussel} and extended $P_4$-laden graphs ($EP_4L$) \cite{laden1} for which their prime factors are a subclass of split graphs, paths and cycles and their complements, and a family of graphs with bounded order. These graph classes generalize several other ones having \emph{few} $P_4$'s such as cographs, $P_4$-tidy, $P_4$-laden and $(7, 3)$ graphs.

\subsection{Preliminaries and notations}

Given $n,m\in \mathbb{N}_0$, $[n,m]$ denotes the set $\{t\in \mathbb{N}_0: n\leq t \leq m\}$ and $[n]=[1,n]$ with the addition modulo $n$. 
Given $a, b \in [n]$, let $t$ be the minimum non-negative integer such that $a + t = b$.
Then, $[a, b]_n$ denotes the circular interval defined by the set $\{a + s : 0 \leq s \leq t\}$. Similarly, $(a, b]_n$, $[a, b)_n$, and $(a, b)_n$
correspond to $[a, b]_n \setminus \{a\}$, $[a, b]_n \setminus \{b\}$, and $[a, b]_n \setminus \{a, b\}$, respectively.

In addition, $K_n$, $S_n$,  $C_n$ ($n\geq 3$), and $P_n$ are, respectively, the complete graph, graph, the cycle, and the path of order $n$.
For all these graphs the set of vertices is $[n]$ and the edge of $C_n$ (resp. $P_n$) are of the form $\{i,i+1\}$ for $i\in [n]$ (resp. for $i\in [n-1]$.)

For a connected graph $G$, the \emph{distance} between two vertices $u,\ v$ in $G$ is the number of edges in a shortest path connecting them. For $m\in\mathbb Z_+$, the $m$-\emph{th power} $G^m$ of $G$ is the graph with the same vertex set that $G$ and two vertices are adjacent in $G^m$ if the distance between them in $G$ is at most $m$.

Given a graph $G$ and $U\subset V(G)$, $G-U$ denotes the graph obtained from $G$ by deleting the vertices in $U$ and $G[U]=G-(V(G)\setminus U)$ denotes the subgraph induced by vertices in $U$. A subset $U$ is a clique of $G$ if $G[U]$ is a complete graph.
 
Given $v\in V(G)$, $N(v)$ denotes the set of neighbors of $v$ 
and $N[v]=N(v)\cup \{v\}$.
A dominating set of $G$ is a subset of vertices $D$ such that every vertex outside of $D$ has a neighbor in $D$, i.e. $D\cap N[v]\neq \emptyset$ for all $v\in V(G)$.

A subset of vertices $I\subset V(G)$ is an independent set of $G$ if no pair of elements in $I$ are adjacent in $G$. The \emph{independence 
number} of $G$ is the maximum cardinality amount its independent sets and it is denoted by $\alpha(G)$. Given a vector of weights $w\in \mathbb{R}^{V(G)}$ and a subset of vertices $U$, its weight is $w(U)=\sum_{v\in U} w_v$.
The maximum weight of an independent set in $G$ is denoted by $\alpha_w(G)$. In the \emph{Maximum Weight Independent Set ($MWIS$) problem}, the input is a graph $G$ and a weight vector $w$ and the output is an independent set $I^*$ of $G$ of weight $\alpha_w(G)$. This problem is $NP$-hard for general graphs and polynomial-time solvable for powers of paths and cycles (see e.g. \cite{clawfree}).

Given $v\in V(G)$ and a graph $H$, the graph obtained by \emph{replacing} $v$ \emph{by} $H$ is the graph with vertex set $(V(G)\setminus\{v\})\cup V(H)$ and whose edges are $E(G-\{v\})\cup E(H)$ together with all the edges connecting a vertex in $V(H)$ with a vertex in $N(v)$. We denote it $G_{v\hookleftarrow H}$. Observe that, if $H$ is the trivial graph with one vertex and no edges, $G_{v\hookleftarrow H}$ is isomorphic to $G$. 

Given $G$ and a family of graphs $\R=\{G_v: v\in V(G)\}$, we denote $G\hookleftarrow \R$ the graph obtained by replacing each vertex $v$ in $G$ by $G_v$, i.e. $\GR$ is the $X$-join product obtained from $G$ and the graphs in $\R$. We say that $G$ is the \emph{main factor} in $\GR$.

A graph $H$ is \emph{prime} for the $X$-join product 
if $H=\GR$ implies either $G$ isomorphic to $H$ and all the graphs in $\R$ are trivial graphs or $G$ is trivial and the graph in $\R$ is isomorphic to $H$.

As it was mentioned before, several operations in graphs can be reformulated in terms of the $X$-join product of graphs.
In particular, the \emph{lexicographic product} of two graphs $G$ and $H$, denoted by $G\circ H$, is the graph with vertex set $V(G)\times V(H)$ and two vertices $(g_1, h_1)$ and $(g_2, h_2)$ are adjacent if either $\{g_1,g_2\}\in E(G)$, or $g_1 = g_2$ and $\{h_1,h_2\}\in E(H)$. It is easy to see that, if $G_v=H$ for all $v\in V(G)$, $G\hookleftarrow \R=G\circ H$. 
Moreover, the disjoint union (resp. the complete join) of disjoint graphs $G$ and $H$, denoted by $G + H$ (resp. $G\vee H$) is the $X$-join product obtained from $S_2$ (resp. $K_2$) and the list $\R=\{G,H\}$.
From the proof of Theorem 2.7 in \cite{martin}, it can be derived that \textit{Gds}'s of $G_1+ G_2$ and $G_1\vee G_2$ can be obtained in constant time from \textit{Gds}'s of $G_1$ and $G_2$.

In the context of this paper we say that a graph is \emph{modular} if it and its complement are connected. 
Clearly, if $G$ is not modular, there exist two disjoint subgraphs $G_1$ and $G_2$ such that $G=G_1+G_2$ or $G=G_1\vee G_2$. 

Moreover, a graph $G$ is called \emph{indecomposable} if for all $U\subset V(G)$ with $2\leq |U|\leq |V(G)|-1$, there exist two vertices $u, v\in U$ such that $N(v)\setminus U \neq N(v)\setminus U$. 
It is not hard to see that, if $G$ is not indecomposable and $U\subset V(G)$ verifies that $N(v)\setminus U =N(u)\setminus U$ for all $u,v\in U$, then for any $u\in U$, $G$ can be obtained from $G-(U\setminus \{u\})$ by replacing $u$ by $G[U]$. 
Then, a graph is indecomposable if and only if it is prime with respect to the $X$-join product. In the following, we will refer to indecomposable graphs as \emph{prime graphs}.

The modular decomposition process applied to a graph $G$ finds in linear time its modular subgraphs and the sequence of disjoint union and complete join operations that have to be performed in order to recover $G$ from them.  

Denoting by $M(\mathcal{F})$ the set of modular graphs in a family of graphs $\mathcal{F}$, we easily obtain the following:
\begin{remark}\label{union}
Let $\mathcal F$ be a family of graphs such that a \textit{Gds} can be obtained in polynomial (resp. linear) time for graphs in $M(\mathcal F)$. Then, a \textit{Gds} can be obtained in polynomial (resp. linear) time for every graph in $\mathcal F$.
\end{remark}

The modular decomposition process continue by finding, also in linear time, for each modular no prime subgraph $H$, a prime subgraph $H'$ of $H$ and a list $\mathcal H=\{H_v: v\in V(H')\}$ of subgraphs of $H$, such that $H= H'\hookleftarrow \mathcal H$. For details on modular decomposition see, e.g. \cite{modular1,Roussel}. 

Given a sequence $S=(v_1,\ldots,v_k)$ of distinct vertices of $G$, the corresponding set $\{v_1,\ldots,v_k\}$ of vertices from the sequence $S$ will be denoted by $\widehat S$ and, for simplicity, $|S|$ represents $|\widehat S|$. The \emph{order} of a vertex $v_i$ in $S=(v_1,\ldots,v_k)$ is $O_S(v_i)=i$. We denote by $S^{-1}$ the sequence $S$ traveled in reverse order, i.e. $S^{-1}=(v_k,v_{k-1},\dots,v_1)$.
Given $U\subset V(G)$, $(U)$ denotes any sequence $S$ such that $\hat S=U$. The \emph{empty sequence} is denoted by $S=(\;)$. 

Let $S=(v_1,\ldots , v_n)$ and $S'=(u_1,\ldots , u_m)$ be two sequences in $G$ with $\widehat S\cap \widehat S'=\emptyset$. The {\em concatenation} of $S$ and $S'$ is defined as the sequence $S \oplus S'=(v_1,\ldots , v_n,u_1,\ldots , u_m)$. Moreover, $(\;)\oplus S= S \oplus (\;)= S$, for any sequence $S$. Clearly $\oplus$ is an associative operation on the set of all sequences, but is not commutative.

By using concatenation of sequences, we can write $S= \bigoplus_{i=1}^n (v_i)$. Following this notation, a \emph{subsequence} of $S$ 
is a sequence of the form %
$\bigoplus_{i=j}^k (v_i)$
with $1\leq j\leq k\leq n$.

If $T=\bigoplus_{i=j}^k (v_i)$ is a subsequence of $S$, $S_{T\hookleftarrow S'}$ is the sequence obtained from $S$ replacing $T$ by $S'$, i.e. $S_{T\hookleftarrow S'}=(v_1,\dots,v_{j-1})\oplus S'\oplus (v_{k+1},\dots,v_n)$. In particular, when $T=(w)$ or $S'=(w)$ 
we denote 
$S_{w\hookleftarrow S'}$ or $S_{T\hookleftarrow w}$, respectively.

In the context of \emph{dominating sequences}, given $S=(v_1,\ldots,v_k)$, 
for all $i\in [k]$, the \emph{private neighborhood of} $v_i$ \emph{with respect to} $S$ is $PN_S(v_i)= N[v_i] \setminus \bigcup_{j=1}^{i-1}N[v_j]$. A sequence $S=(v_1,\ldots,v_k)$ is a {\em legal dominating sequence of} $G$ if $\widehat{S}$ is a dominating set of $G$ and $PN_S(v_i) \ne\emptyset$ holds for every $i\in [k]$.
We denote by $\mathcal L (G)$ the set of all legal dominating sequences of $G$. 
A \emph{Grundy dominating sequence} (\textit{Gds} for short) is a legal dominating sequence of maximum length. We denote by $\mathcal{G}r(G)$ the set of all \textit{Gds}'s of $G$.  The \emph{Grundy domination number} of $G$ is the length of a \textit{Gds} in $G$ and it is denoted by $\gr(G)$.

The next result is a straightforward counting observation.

\begin{remark}\label{cotalegal}
Let $G$ be a graph,  $S=(v_1,\ldots, v_k)\in \mathcal L(G)$ and $R\subset [k]$. Then, $|S|-|R|\leq |V(G)|-\left|\bigcup_{i\in R} PN_S(v_i)\right|$. In particular, $\forall r\in [k]$, if $R=[r]$, $\bigcup_{i\in R} PN_S(v_i)=\bigcup_{i=1}^r N[v_i]$ and then, $|S|\leq|V(G)|-\left|\bigcup_{i=1}^r N[v_i]\right|+r$. 
\end{remark}

Given $S\in \mathcal L(G)$, we say that $v\in  \widehat S$ is \emph{the footprinter} of 
every vertex in $PN_S(v)$.
Clearly, every vertex in $V(G)$ has a unique footprinter and the function $f_S : V(G)\rightarrow \widehat S$ that maps each vertex to its footprinter is well defined. Moreover,  defining $I_S=\{v\in V: f_S(v)=v\}$, it is easy to see the following:
\begin{remark}\label{indgr}
For every $S\in \mathcal L(G)$, $I_S$ is an independent set of $G$.
\end{remark}
 
Moreover, given an independent set $I$ of $G$, we define $\mathcal L (G,I)=\{ S\in \mathcal L(G): I_S=I\}$ and considering that the maximum of the empty set is $-\infty$, $\gr(G,I)=\max\{|S|: S\in \mathcal L(G,I)\}$. We have 

\begin{remark}\label{gdsI}
$\gr(G)=\max \{\gr(G,I): I \text{ is an independent set of } G\}$.
\end{remark}

%


\section{\textit{Gds}'s on $X$-join product of graphs}

In this section we are interested in general properties of legal dominating sequences of $X$-join product of graphs. 

Given a graph $G$, $\R=\{G_v: v\in V(G)\}$, $S=(w_1,\ldots, w_n)$ a sequence of vertices of $G$, and $v\in V(G)$, we define $n_S(v)=|\widehat S \cap V(G_v)|$. If $n_S(v)\geq 1$, we denote as $\ell_S(v)$ the vertex in $ \widehat S\cap V(G_v)$ with the lowest order in $S$.

\begin{lemma}\label{almenos}
Let $G$ be a graph and $\R=\{G_v: v\in V\}$. Then, 
\begin{equation}\label{grGI}
\gr(\GR)\geq \max \left\{\gr(G,I)+\sum_{v\in I} \gr(G_v)-|I|: I \text{ is an independent set of } G\right\}.
\end{equation}
\end{lemma}

\begin{proof}
Let $I$ be an independent set of $G$. We need to prove that
$$\gr(\GR)\geq \gr(G,I)+\sum_{v\in I} \gr(G_v)-|I|.$$

If $\mathcal L(G,I)=\emptyset$, the inequality follows immediately. 

Otherwise, for each $v\in I$, let $S_v\in \mathcal Gr(G_v)$ and, if $v\in V(G)\setminus I_S$, let $S_v=(w_v)$ for some $w_v\in V(G_v)$.

Given $S\in \mathcal L(G,I)$ such that $|S|=\gr(G,I)$, let $\tilde S$ the sequence obtained by replacing each $v\in I$ by $S_v$. Clearly, $\tilde S\in \mathcal L(\GR)$. Then, $\gr(\GR)\geq |\tilde S|= \gr(G,I)+\sum_{v\in I} \gr(G_v)-|I|$.
\end{proof}

We will see that the equality holds in (\ref{grGI}). We first need the following technical result.

\begin{lemma}
Let $S\in \mathcal Gr(\GR)$. Then, there exists $S'\in \mathcal Gr(\GR)$ verifying that, for all $v\in V(G)$ such that $\ell_{S'}(v)\in I_{S'}$, it holds: 

\begin{enumerate}
	\item \label{first} $f_{S'}(w)\in V(G_v)$ for all $w\in V(G_v)$ and
	\item \label{second} $|\widehat{S'} \cap V(G_v)|=\gr (G_v)$.
\end{enumerate}
\end{lemma}

\begin{proof}
In order to prove \ref{first}, it is enough to show that if there exist $v \in V(G)$ and $w\in V(G_v)$ such that $f_S(\ell_S(v))=\ell_S(v)$ and $f_S(w)\notin V(G_v)$ then there exists $S'\in \mathcal Gr(\GR)$ such that  
$f_{S'}(w)\in V(G_v)$ and $I_S\subset I_{S'}$.

Let $S'=S_{f_S(w)\hookleftarrow w}$. It is not hard to see that, for all $z\in \widehat{S'}$, $PN_{S'}(z)\neq \emptyset$. Since $S\in \mathcal Gr(\GR)$ and $|S'|=|S|$, $S'\in Gr(\GR)$.

Now, let us prove that, if $S'$ verifies \emph{\ref{first}}, $|\widehat{S'} \cap V(G_v)|=\gr (G_v)$ for all $v\in V(G)$ with $\ell_{S'}(v)\in I_{S'}$. 

Since $S'$ verifies \emph{\ref{first}}, for each $v\in V(G)$ such that $\ell_{S'}(v)\in I_{S'}$, there exists a subsequence $S_v\in \mathcal L(G_v)$ with $\widehat{S'_v}=\widehat{S'} \cap V(G_v)$. 

Moreover, for every $\tilde S_v\in \mathcal Gr(G_v)$, $S'_{S_v\hookleftarrow \tilde S_v}\in \mathcal L(\GR)$. Since $S'\in \mathcal Gr(\GR)$, $|S_v|=|\tilde S_v|=\gr(G_v)$.

%
\end{proof}

Finally, we can prove:

\begin{theorem}\label{gdsgr}
Let $G$ be a graph and $\R=\{G_v: v\in V\}$. Then, 
$$\gr(\GR)=\max \left\{\gr(G,I)+\sum_{v\in I} \gr(G_v)-|I| : I \text{ is an independent set of } G\right\}.$$
\end{theorem}

\begin{proof}
Let $S\in \mathcal Gr(\GR)$. We need to prove that there exists an independent set $I$ of $G$ such that 
$$|S|=\gr(\GR)\leq \gr(G,I)+\sum_{v\in I} \gr(G_v)-|I|.$$
Let $I=\{v\in V(G): f_S(\ell_S(v))=\ell_S(v)\}$. Clearly, $I$ is an independent set of $G$.

W.l.o.g. we can assume that $S$ verifies the conditions of Lemma \ref{almenos} and, for all $v\in V(G)$, there is a subsequence $S_v$ of $S$ such that $\hat S_v=\hat S \cap V(G_v)$. Moreover, for all $v\in I$, 
$|S_v|=\gr(G_v)$ and, if $n_S(v)\geq 1$ and $v\notin I$, $n_S(v)=1$.

Let $S_G$ be the sequence of $G$ obtained from $S$ by replacing each subsequence $S_v$ with $|S_v|\geq 1$ by $(v)$

We will prove that $S_G\in \mathcal L(G,I)$.

Let $v\in \widehat{S_G}$. Note that, $PN_S(\ell_S(v))\cap V(G_v)=\emptyset$ if and only if $v\notin I$. Besides, for $u\neq v$, $PN_S(\ell_S(v))\cap V(G_u)$ is $V(G_u)$ or the empty-set. Hence, $PN_{S_G}(v)=\left\{u\in V(G):\ PN_{S_G}(v)\cap V(G_u)\neq\emptyset\right\}$. This remark shows that $\tilde S\in \mathcal L(G,I)$.

Then, 
$$|S|=|S_G|+\sum_{v\in I} \gr(G_v)-|I|\leq \gr (G,I)+\sum_{v\in I} \gr(G_v)-|I|$$
and the proof is complete.
\end{proof}

The following lemma can be easily obtained as a corollary of Theorem \ref{gdsgr} and gives a direct formula for the case of the replacement of one vertex by a graph.

For every $u\in V(G)$, we define $\gr^u(G)=\max\{\gr(G,I): u\in I\}$.
\begin{lemma}\label{remplazovertice}
Let $G$ and $H$ be two disjoint graphs and $u\in V(G)$. 
Then, $\gr(G_{v\hookleftarrow H})=\max \{\gr(G), \gr^u(G)+\gr(H)-1\}$.
\end{lemma}

In the next two sections we will study the \textit{Gds}'s of $\GR$ when $G=P_n^m$ or $G=C_n^m$, applying Theorem \ref{gdsgr}. Consequentially, we will analyze $\gr(G,I)$, for every independent set $I$ of $G$. In this regard, the next lemma gives us a tool to study this parameter for $G=P_n^m$ or $G=C_n^m$.

\begin{lemma}\label{entrepotcaminos}
Let $G=P_n^m$ or $G=C_n^m$, $i\in [n]$ and $m+1\leq t\leq n$. Let $S\in\mathcal L(G)$ such that $I_S\cap [i,i+t]_n=\{i,i+t\}$ . Then,

$$\left|\widehat S\cap [i,i+t)_n\right|\leq t-m.$$
\end{lemma}

\proof
Let $v\in (i,i+t)_n$. By hypothesis, $f_S(v)\neq v$ and it is not hard to see that $f_S(v)\in [i,i+t]_n$. 

Let $F_+=\{v\in \widehat S\cap (i, i+t)_n: \ f_S(v)\in [i,v)_n\}$ and $F_-=\{v\in \widehat S \cap (i, i+t)_n:\ f_S(v)\in (v,i+t]_n\}$. We have $F_+\cap F_-=\emptyset$ and $\widehat S \cap (i, i+t)_n=F_+\cup F_-$. We need to prove that $|F_+|+|F_-|\leq t-m-1$

The result is trivial if $F_+=F_-=\emptyset$.  

Assume that $F_+\neq\emptyset$. Let $v$ be the vertex in $F_+$ at minimum distance from $i$. Then, $i=f_S(v)$. Let us show that $[i,i+m]_n\subset PN_S(i)$. First, it is easy to see that $O_S(i)<O_S(u)$ for all $u\in F_+$. Hence assume that there exists $w\in F_-$ such that $PN_S(w)\cap [i,i+m]_n\neq\emptyset$. Then, $O_S(w)<O_S(i)<O_S(v)$, which is a contradiction since $N[v]\subset N[i]\cup N[w]$. So, $[i,i+m]_n\subset PN_S(i)$. Therefore $\left|\widehat S \cap [i+1, i+t-1]_n\right|$ is at most $\left|(i+m,i+t)_n\right|=t-m-1$ and the result follows.

Analogous reasoning is valid if $F_-\neq\emptyset$.
%
%
%
%
%
\qed

\section{\textit{Gds}'s on $X$-join product with a power of a cycle as main factor}

Observe that $C_n^m$ with $2(m+1)> n$ is isomorphic to $K_n$ and in this case $\GR$ is the complete join of graphs in $\R$. Hence, from the results in Theorem 2.7 in \cite{martin}, the \textit{Gds}'s of the graphs in $\R$ with maximum Grundy domination number are \textit{Gds}'s of $\GR$. Then, from now on we assume that, if $G=C^m_n$, $2(m+1)\leq n$. 
Recall that we assumme that $V(C_n^m)$ is $[n]$ with the addition modulo $n$ and $E(C_n)=\{\{i, i+1\}: i=1\ldots,n\}$.

\medskip

Given $i,j$ such that 
$1\leq i\leq j\leq n$, let $S(i,j)=\bigoplus_{t=0}^{j-i}(i+t)$. 

\bigskip

We denote $\mathcal I_2$ the family of independent sets of $C_n^m$ with cardinality at least two. Note that $\mathcal I_2\neq\emptyset$ since $\{1,m+2\}\in\mathcal I_2$.

Given $I=\{i^j: j\in [p]\}\in \mathcal I_2$ with $1\leq i^j < i^{j+1}\leq n$, for $j\in[p-1]$, let $S_j= S(i^j,i^{j+1}-(m+1))$, $j\in [p-1]$. Moreover, we define $S_p=S^{-1}(i^p+(m+1),i_1-1)$ if $i^p+(m+1)\neq i^1$ and $S_p=(\;)$, otherwise. Finally, let $$S_C(I)=\left(\bigoplus_{j=1}^{p-1}S_j\right)\oplus S_p \oplus (i^p).$$

It is not hard to check that, for all $I\in \mathcal I_2$,  $S_C(I)\in \mathcal L (G,I)$ and $|S_C(I)|=n- |I|m$.

For $I=\{k\}\subset [n]$ we define $S_C(I)=S_C(\{k,k+m+1\})$. Observe that $\{k,k+m+1\}\in \mathcal I_2$. We have the following result:

\begin{lemma}\label{SCI}
For any non empty independent set $I$ of $C_n^m$, $\gr(C_n^m,I)=|S_C(I)|$. Then, $\gr(C_n^m,I)= n- |I|m$ if $I\in \mathcal I_2$ and $\gr(C_n^m,I)= n-2m=|S_C(I)|$ if $|I|=1$.
\end{lemma}

\proof
Let $I=\{i^j: j\in [p]\}$ with $1\leq i^j < i^{j+1}\leq n$, for all $j\in [p-1]$, an independent set of $C_n^m$ and $S \in \mathcal L (G,I)$. We only need to prove that $|S|\leq |S_C(I)|$.

First assume that $|I|\geq 2$ and, for $j\in [p]$ denote $I^j=[i^j,i^{j+1})_n$. From Lemma \ref{entrepotcaminos} we have that:

\begin{enumerate}
	\item for any $j\in [p-1]$, $\left|\widehat S\cap I^j \right|\leq (i^{j+1}-i^j)-m$ and 
	\item $\left|\widehat S\cap I^p \right|\leq (i^1-i^{p})+n-m$. 
\end{enumerate}

Then,

$$|S|=\left(\sum_{j\in [p]}\left|\widehat S\cap	I^j \right|\right) \leq\left(\sum_{j\in [p]}(i^{j+1}-i^j)-m\right)+n =n-pm=|S_C(I)|.$$
	
Finally, let $I=\{k\}$. We need to prove that 
$$|S|\leq |S_C(\{k,k+m+1\})|= n-2m.$$ 

Clearly, $k$ is the first vertex in the sequence $S$ and footprints $2m+1$ vertices. Then, from Remark \ref{cotalegal} $|S|\leq 1+(n-2m+1)=n-2m$ and the theorem holds.
\qed

\medskip

As a direct consequence of previous theorem and Theorem \ref{gdsgr} we obtain:

\begin{theorem}\label{grundyciclos}
Let $G=C^m_n$, $\R=\{G_i:i\in [n]\}$. Then

$$\gr(\GR)= \max_{I \in \mathcal I_2} \left\{\sum_{i\in I} \gr (G_i)- |I|(m+1)\right\} + n.$$

Moreover, given $S_i\in \mathcal Gr(G_i)$ for all $i\in [n]$ and 
$I^*\in \mathcal I_2$ such that

$$\sum_{i\in I^*} \gr (G_i)- |I^*|(m+1)=\max_{I \in \mathcal I_2} \left\{\sum_{i\in I} \gr (G_i)- |I|(m+1)\right\},$$
the sequence $S$ obtained from $S_C(I^*)$ by replacing each $i\in I^*$ by $S_i$ verifies $S\in \mathcal Gr (\GR)$.
\end{theorem}

Recalling that $G\circ H= \GR$ with $\mathcal R= \{G_v=H : v\in V(G)\}$, from the last result we have: 

\begin{theorem}\label{grlexcycle}
Let $n,m\in \Z_+$ and $H$ be a graph. Then,

$$\gr(C_n^m\circ H)=\left\{
	\begin{array}[h]{lll}
			\left\lfloor \frac{n}{m+1}\right\rfloor (\gr (H)-(m+1))+n & \textnormal{ if } & \gr(H)\geq m+1,\\
		&&\\
		2\gr(H)+n-(2m+2) & \textnormal{ if } & \gr(H)\leq m.
	\end{array}\right.$$
\end{theorem}

\proof
Applying Theorem \ref{grundyciclos} we have:
$$\gr(C_n^m\circ H)=\max_{I \in \mathcal I_2} \left\{|I| [\gr (H)-(m+1)]\right\} + n.$$

If $\gr(H)\geq m+1$, 

$$\max_{I \in \mathcal I_2} \left\{|I| [\gr (H)-(m+1)]\right\}= \max_{I \in \mathcal I_2} \left\{|I|\right\} [\gr (H)-(m+1)]=$$
$$=\alpha(C_n^m)[\gr (H)-(m+1)].$$
Then,
$$\gr(C_n^m\circ H)=\left\lfloor \frac{n}{m+1}\right\rfloor [\gr (H)-(m+1)]+n.$$

If $\gr(H)\leq m$, 
$$\max_{I \in \mathcal I_2} \left\{|I| [\gr (H)-(m+1)]\right\}= \min_{I \in \mathcal I_2} \left\{|I|\right\} [\gr (H)-(m+1)]=2 [\gr (H)-(m+1)].$$ 
Then, $$\gr(C_n^m\circ H)=2\gr(H)+n-(2m+2).$$
\qed

Observe that, by fixing $m=1$ we derive the known formula for the lexicographic product $C_n\circ H$ obtained in \cite{Bresar1} for $\gr(H)\geq 2$. Moreover, applying Theorem \ref{grlexcycle} with $H$ the trivial graph with one vertex, the Grundy domination number for $C_n^m$ is obtained.

\begin{corollary}
$\gr(C_n^m)=n-2m$.
\end{corollary}

In this case, fixing $m=1$ we derive the known formula for the Grundy domination number of $C_n$.

\medskip

Note that, given $\gr(H)$, the Grundy domination number of $C_n^m\circ H$ can be computed in constant time and the same result holds for $\GR$, if $G=C_n^m$ and $\gr(G_v)$ is a constant for all $G_v\in \R$. However, the computational complexity of computing $\gr(\GR)$ for a general family $\R$ is not so clear. Next, we show that 
this problem can be reduced to the Maximum Weight Independent Set problem on power of cycles.

Let $G=C^m_n$ and $\R=\{G_i:i\in [n]\}$, and 
$$M=\max_{I \in \mathcal I_2} \left\{\sum_{i\in I} \gr (G_i)- |I|(m+1)\right\}=\max_{I \in \mathcal I_2} \left\{\sum_{i\in I} [\gr (G_i)- (m+1)]\right\}.$$

From Theorem \ref{grundyciclos}, we know that if $I^*\in \mathcal I_2$ is an independent set where this maximum is attained, $S_C(I^*)$ can be constructed in linear time. Moreover, given $S_i\in \mathcal Gr(G_i)$ for all $i\in [n]$, a \textit{Gds} of $\GR$ can be constructed in linear time.

Then, we need to analyze the computational complexity of computing $M$. Defining the vector of weights $w\in \Z^n$ such that $w_i= \gr (G_i)- (m+1)$ the problem can be reduced to obtain a maximum weighted independent set of $C_n^m$ with cardinality at least two. Let us analyze the relationship between $M$ and $\alpha_w(C_n^m)$.

Let $I^*$ be an independent set of $C_n^m$ such that $w(I^*) = \alpha_w(C_n^m)$. Clearly, if $|I^*|\geq 2$, $M=\alpha_w(C_n^m)$.  

We will see that, if $|I^*|\leq 1$, $M=\alpha^2_w(C_n^m)=\max_{I \in \mathcal I_2} \left\{ w(I): |I|=2\right\}$ or, equivalently, for all $I\in \mathcal I_2$ with $|I|\geq 3$ there exists $j\in I$ such that $w(I\setminus \{j\})\geq w(I)$.

Clearly, if $|I^*|=0$ then $w_i\leq 0$ for all $i\in[n]$. Then, given $I\in \mathcal I_2$ with $|I|\geq 3$, $w(I\setminus \{j\})\geq w(I)$ for all $j\in I$. 

If $|I^*|=1$ and $I^*=\{k\}$, we have $w_k>0$ and $w_j\leq 0$ for all $j\in [k+m+1, k-m-1]_n$. Otherwise, if $w_j>0$ for some $j\in [k+m+1, k-m-1]_n$, $\{k,j\}$ is a independent set with $w_k+w_j> w(I^*)$.

Let $I\in \mathcal I_2$ with $|I|\geq 3$. Then, there exists $j\in I\cap [k+m+1, k-m-1]_n$ and $I\setminus \{j\}\in \mathcal I_2$ with $w(I\setminus \{j\})\geq w(I)$.

Since the MWIS problem is polynomial time solvable in powers of cycles \cite{clawfree}, we have proved:

\begin{theorem}\label{pesosnoneg}
Let $G=C_n^m$ and $\R=\left\{G_i:i\in [n]\right\}$,   
$w\in \Z^n$ with $w_i= \gr(G_i)-(m+1)$ for $i\in [n]$ and $I^*$ such that $w(I^*)=\alpha_w(C_n^m)$. 
Then, if $|I^*|\geq 2$, $\gr(\GR)= \alpha_w(C_n^m)+n$. Otherwise, $\gr(\GR)= \alpha^2_w(C_n^m)+n$.
Therefore, given $S_i\in \mathcal Gr(G_i)$ for all $i\in [n]$, a \textit{Gds} of $\GR$ can be obtained in polynomial time. 
\end{theorem}


\section{\textit{Gds}'s on $X$-join product with a power of a path as main factor}

Let us now consider power of paths. Observe that $P_n^m$ with $m+2>n$ is isomorphic to $K_n$.
Then, from now on, if $G=P^m_n$ we assume that $m+2\leq n$.
Recall that we assume that $V(P_n)=[n]$ and $E(P_n)=\{\{i,i+1\}: i\in [n-1]\}$.

Let us denote by $\mathcal I$, the family of non-empty independent sets of $P_n^m$ and $\tilde{\mathcal I}=\{I\in \mathcal I: m(I)\leq m+1 \text{ and } M(I)\geq n-m\}$. Similarly as in the case of power of cycles, for each $I\in \tilde{\mathcal I}$ we associate a legal dominating sequence $S_P(I)$ of $G$.

Let $I=\{i^j: j\in [p]\}\in \tilde{\mathcal I}$ with $1\leq i^j < i^{j+1}\leq n$ for all $j\in [p-1]$ and  $S_j=S(i^j,i^{j+1}-(m+1))$, for $j\in[p-1]$. We define:

$$S_P(I)=\left(\bigoplus\limits_{j=1}^{p-1} S_j\right)\oplus (i^p).$$

It is not hard to verify that $S_P(I)\in \mathcal L(P_n^m, I)$ for all $I\in \tilde{\mathcal I}$. 
Therefore, $\gr(P_n^m, I) \geq 
\max\limits_{I\in \tilde{\mathcal I}}|S_P(I)|$. 

Observe that, if $m(I)=\min\{i:\ i\in I\}$ and $M(I)=\max\{i:\ i\in I\}$, $|S_P(I)|=M(I)-m(I)+1-(|I|-1)m$.    

We have the following result:

\begin{lemma}\label{SI}
Let $I\in \tilde{\mathcal I}$. Then, $$\gr(P_n^m,I)=|S_P(I)|=M(I)-m(I)+1-(|I|-1)m.$$
Besides, if $I\in\mathcal I$, there exists $I' \in \tilde{\mathcal I}$ such that $\gr(P_n^m,I)\leq\gr(P_n^m,I')$.
\end{lemma}

\proof
Let $I=\{i^j: j\in [p]\}$ with $1\leq i^j < i^{j+1}\leq n$, for all $j\in [p-1]$, be an independent set. For $j\in [p-1]$ denote $I^j=[i^j,i^{j+1})$. Moreover, $I^0=[0,i^1)$ and $I^p=[i^p,n]$.

From Lemma \ref{entrepotcaminos} we have that for any $S \in \mathcal L (P_n^m, I)$ and $j\in [p-1]$, $\left|\widehat S\cap I^j \right|\leq (i^{j+1}-i^j)-m$.

Let $I\in \tilde{\mathcal I}$, i.e. $i_1\leq m+1$ and $i_p\geq n-m$. Clearly, $\gr(P_n^m, I) \geq |S_P(I)|=M(I)-m(I)+1-(|I|-1)m$.
Then, we need to prove that, if $S\in \mathcal L(P_n^m, I)$, $|S|\leq M(I)-m(I)+1-(|I|-1)m$.

Let $S \in \mathcal L (P_n^m, I)$. Since $i^1\leq m+1$, $i^1\in N[v]\subset N[i^1]$ for all $v\in I^0$. Then $I^0\cap\widehat S=\emptyset$. Analogous reasoning implies that $I^p\cap\widehat S=\{v^p\}$.

Therefore, $|S|=\sum_{j=0}^p|\widehat S\cap I^j|\leq1+\sum_{j=1}^{p-1}((i^{j+1}-i^j)-m)=M(I)-m(I)+1-(|I|-1)m$.

For the second part, it is enough to prove that,

\begin{enumerate}
	\item if $i^1\geq m+2$,  $\gr(P_n^m,I)\leq \gr(P_n^m,I\cup\{1\})$ and 
	\item if $i^p \leq n-m-1$, $\gr(P_n^m,I)\leq \gr(P_n^m,I\cup\{n\})$. 
\end{enumerate}

%

Consider the case $i^1\geq m+2$. 

Let $S \in \mathcal L (P_n^m, I)$. Note that, if $v\in I^0\cap\widehat S$, $O_S(v)> O_S(i^1)$. Otherwise, 
the lowest order vertex in $I^0\cap\widehat S$ footprints itself, a contradiction considering that $I\cap I^0=\emptyset$.

Hence, if $v\in I^0\cap\widehat S$, then $PN_S(v)\subset I^0\setminus [i^1-m,i^1)$ and $|I^0\cap\widehat S|\leq i^1-1-m$. 
As above, for any $j\in [p-1]$, $\left|\widehat S\cap I^j \right|\leq (i^{j+1}-i^j)-m$. 

Therefore, $|S|\leq |S_P(I\cup\{1\})|$. We have proved that $\gr(P_n^m,I)\leq |S_P(I\cup\{1\})|\leq \gr(P_n^m,I\cup\{1\})$.

Let us now analyze the case $i^p\leq n-(m+1)$. With a similar reasoning followed in the previous case, we have that $|I^p\cap\widehat S|\leq n+1-i^p-m$ and $|S|\leq |S_P(I\cup\{n\})|$ for every $S\in \mathcal L(G,I)$.  Then, 
$\gr(P_n^m,I)\leq |S_P(I\cup\{n\})|\leq \gr(P_n^m,I\cup\{n\})$.
\qed

From the previous lemma and Theorem \ref{gdsgr}, we have:

\begin{theorem}\label{grundypotcaminos}
Let $G=P^m_n$ and $\R=\{G_i:i\in [n]\}$.
Then,
$$\gr(\GR)=\max_{I \in \tilde{\mathcal I}} \left\{\sum_{i\in I} (\gr (G_i)-1)+ M(I)-m(I)-|I|m\right\}+m+1.$$

Moreover, given $S_i\in \mathcal Gr(G_i)$ for all $i\in [n]$ and 
$I^*\in \tilde{\mathcal I}$ such that $$\sum_{i\in I^*} (\gr (G_i)-1)+M(I^*)-m(I^*)- |I^*|m+m+1=\gr(\GR),$$

the sequence $S$ obtained by replacing in $S_P(I^*)$ each $i\in I^*$ by $S_i$ verifies $S\in \mathcal Gr (\GR)$.
\end{theorem}

\begin{proof}
We only need to verify that if $I\in \mathcal I$ and $m(I)\geq m+2$ (resp. $M(I) \leq n-m-1$), defining $\tilde I=I\cup \{1\}$ (resp. $\tilde I=I\cup \{n\}$) 
$$\sum_{i\in I} (\gr (G_i)-1)+ M(I)-m(I)-|I|m \leq \sum_{i\in \tilde I} (\gr (G_i)-1)+ M(\tilde I)-m(\tilde I)-|\tilde I|m$$
\end{proof}

Recalling that $G\circ H= \GR$ with $\mathcal R= \{G_v=H : v\in V(G)\}$, we have:

\begin{theorem}\label{grlexpath}
Let $n,m\in \Z_+$ and $H$ be a graph. Then,

	$$\gr(P_n^m\circ H)=\left\{
	\begin{array}[h]{lll}
		\left\lceil \frac{n}{m+1}\right\rceil (\gr (H)-(m+1))+n+m & \textnormal{ if } & \gr(H)\geq m+1,\\
		&&\\
		2\gr(H)+n-m-2 &  \textnormal{ if } & \gr(H)\leq m.
	\end{array}\right.$$
\end{theorem}

\proof
From Theorem \ref{grundypotcaminos} we have that
$$\gr(P_n^m\circ H)=\max_{I \in \tilde{\mathcal I}} \left\{|I|[\gr (H)- (1+m)]+ M(I)-m(I)\right\}+(m+1).$$
First, in order to obtain $\gr(P_n^m\circ H)$, we note that it is enough to consider independent sets $I$ such that $M(I)=n$ and $m(I)=1$. Indeed, if $\ell(I)=|I|[\gr (H)- (1+m)]+ M(I)-m(I)$ and $\tilde I= (I\setminus \{M(I),m(I)\})\cup \{1,n\}$, it is easy to see that $\ell(I)\leq \ell(\tilde I)$. Then, 
$$\gr(P_n^m\circ H)=\max_{I \in \tilde{\mathcal I}} \left\{|I|[\gr (H)- (m+1)]\right\}+ m+ n.$$

Therefore, if $\gr(H)\geq m+1$, 
$$\max_{I \in \tilde{\mathcal I}} \left\{|I|[\gr (H)- (1+m)]\right\}=\max \left\{|I|: I \in \tilde{\mathcal I}\right\} [\gr (H)- (1+m)]=$$
$$= \alpha(P_n^m) [\gr (H)- (1+m)].$$
Since $\alpha(P_n^m)=\left\lceil \frac{n}{m+1}\right\rceil$, 
$$\gr(P_n^m\circ H)=\left\lceil \frac{n}{m+1}\right\rceil (\gr (H)-(m+1))+n+m.$$

Finally, if $\gr(H)\leq m$, 
$$\max_{I \in \tilde{\mathcal I}} \left\{|I|[\gr (H)- (1+m)]\right\}= \min \left\{|I|:I \in \tilde{\mathcal I}\right\} [\gr (H)- (1+m)].$$
Then, $$\gr(P_n^m\circ H)=2(\gr(H)-(m+1))+n+m=2\gr(H)+n-m-2.$$
\qed

Observe that, by fixing $m=1$ we derive the known formula for the lexicographic product $P_n\circ H$ obtained in \cite{Bresar1} for $\gr(H)\geq 2$. Moreover, applying Theorem \ref{grlexpath} with $H$ the trivial graph with one vertex, the Grundy domination number for $P_n^m$ is obtained.

\begin{corollary}
$\gr(P_n^m)=n-m$.
\end{corollary}

In this case, fixing $m=1$ we derive the known formula for the Grundy domination number of $P_n$.

Note that, the Grundy domination number of $P_n^m\hookleftarrow\mathcal R$ can be computed in constant time if $\gr(G_v)$ is a given constant, for all $G_v\in \R$. Let us analyze the computational complexity for general families $\R$.
From Theorem \ref{grundypotcaminos}, it depends on the computational complexity of computing 
$$M= \max_{I \in \tilde{\mathcal I}} \left\{\sum_{v\in I} \gr (G_v)-(|I|-1) (1+m)+ M(I)-m(I)\right\}.$$
Given $I \in \tilde{\mathcal I}$, 
we define $I_1= I\cap [1,m+1]$, $I_2=I\cap [m+2,n-m-1]$ and $I_3=I\cap [n-m,n]$. Observe that, since $I \in \tilde{\mathcal I}$, 
$$m(I)=\sum_{i\in I_1} i \text{ \;\; and \;\; } M(I)=\sum_{i\in I_3}i.$$ 

Then,if $|I|\geq 2$ we have

$$\sum_{i\in I} \gr (G_i)-(|I|-1) (1+m)+ M(I)-m(I)= $$
$$=\sum_{i\in I_1}[\gr (G_i)-i]+ \sum_{i\in I_2} [\gr (G_i)- (1+m)]+ \sum_{i\in I_3} [\gr (G_i)- (1+m)+i].$$

Observe that, if $n\geq 2m+3$, $|I|\geq 2$ for all $I\in \tilde I$. However, if $n\leq 2m+2$, $n-m\leq m+2$. 
Then, for any $j\in [n-m,m+1]$, 
$\{j\}\in\tilde{\mathcal I}$. Moreover, if $I\in\tilde{\mathcal I}$ and $|I|=1$, $I\subset (n-m-1,m+2)$.

Let us first analyze the case $n\geq 2m+3$. 

Since $|I|\geq 2$ for all $I\in \tilde I$, we have
$$M=\max_{I\in \tilde{\mathcal I}}\left\{\sum_{i\in I_1}[\gr (G_i)-i]+ \sum_{i\in I_2} [\gr (G_i)- (1+m)]+ \sum_{i\in I_3} [\gr (G_i)-(1+m)+i] 
\right\}.$$

Let $w\in\Z^n$ a weight vector defined as follows:
\begin{equation} \label{pesos}
w_i=\left\{
\begin{array}{lll}
\gr (G_i)-i& \text{ if } & i\in [1,1+m] \\
\gr (G_i)-(1+m)& \text{ if } & i\in [m+2,n-m-1] \\
\gr (G_i)-(1+m)+i& \text{ if } & i\in [n-m,n] 
\end{array}
\right.
\end{equation}

We will see that $M=\alpha_w(P_n^m)$. We only need to prove that there always exists $I^*\in \tilde I$ such that $w(I^*)=\alpha_w(P_n^m)$.

Observe that $w_1\geq 0$ and $w_n > 0$. Then, if $I^*$ is an independent set of $P_n^m$ such that $w(I^*)=\alpha_w(P_n^m)$, $M(I^*)\geq n-m$. Moreover, if $m(I^*)\geq m+2$ then $w_j\leq 0$ for all $j\in [m(I^*)-(m+1)]$. Hence, $w_1=0$ and $w(I^*\cup \{1\})=w(I^*)$. Clearly, $I^*\cup \{1\}\in \tilde{\mathcal I}$ and then,  $M=\alpha_w(P_n^m)$.

Since the MWIS problem is linear time solvable in powers of paths \cite{cordales}, we have proved:

\begin{lemma}
Let $G=P_n^m$ with $n>2m+2$, $\R=\left\{G_i:i\in [n]\right\}$. Let $w\in \Z^n$ defined as in (\ref{pesos}).
Then, $\gr(\GR)=\alpha_w(P_n^m)$. Moreover, given $S_i\in \mathcal Gr(G_i)$ for all $i\in [n]$, a \textit{Gds} of $\GR$ can be obtained in linear time.
\end{lemma}

Let us analyze the case $n\leq 2m+2$. Since $\alpha(P_n^m)=2$, $\gr(\GR)$ can be computed in $O(n^2)$ exploring all the elements in $\tilde I$. However, we will see that, also in this case, computing $M$ can be reduced to the $MWIS$ in $P_n^m$ and then, it 
can be solved in linear time.

Recall that $j\in [n-m,m+1]$ if and only if  $\{j\}\in \tilde I$. Moreover, if $I\in \tilde I$ and $|I|=2$, $|I\cap [1,n-m-1]|=|I\cap [m+2,n]|=1$.

Then, given $I \in \tilde{\mathcal I}$ if $I_1= I\cap [1,n-m-1]$, $I_2=I\cap [n-m,m+1]$, and $I_3=I\cap [m+2,n]$ we have:


$$\sum_{i\in I} \gr (G_i)-(|I|-1) (1+m)+ M(I)-m(I)=$$
$$=\sum_{i\in I_1} \gr (G_i)+\sum_{i\in I_2} \gr (G_i)+\sum_{i\in I_3} \gr (G_i)-(|I|-1) (1+m)+ \sum_{i=n-m}^n i-\sum_{i=1}^{m+1} i=$$
$$=\sum_{i\in I_1} [\gr (G_i)-i]+\sum_{i\in I_2} \gr (G_i)+\sum_{i\in I_3} [\gr (G_i)+i]-(|I|-1)(1+m).$$

Then, using a similar reasoning as before, $M=\alpha_w(P_n^m)$ with $w\in\Z^n$ a weight vector defined as follows:

\begin{equation} \label{pesos1}
w_i=\left\{
\begin{array}{lll}
\gr (G_i)-i& \text{ if } & i\in [1,n-m-1] \\
\gr (G_i)& \text{ if } & i\in (n-m-1,m+2) \\
\gr (G_i)+i-(1+m)& \text{ if } & i\in [m+2,n] 
\end{array}
\right.
\end{equation}



\begin{theorem}
Let $G=P_n^m$ with $n\leq 2m+2$, $\R=\left\{G_i:i\in [n]\right\}$. Let $w\in \Z^n$ defined as in (\ref{pesos1}).
Then, $\gr(\GR)=\alpha_w(P_n^m)$. Moreover, given $S_i\in \mathcal Gr(G_i)$ for all $i\in [n]$, a \textit{Gds} of $\GR$ can be obtained in linear time.
\end{theorem}

%
%
%
%
%
%
%
%
%
%
%
%
%
%
%
%
%
%
%
%


\section{\textit{Gds}'s on $X$-join product with a split graph as main factor}

In this section we work with split graphs $G=(I^*\cup K, E)$ where $K$ is a clique of $G$ and $I$, an independent set with $|I^*|=\alpha(G)$.

Given a split graph $G$ we define the parameter $n(G)=1$ if there exist $v,w\in K$ such that $(N(v)\cap N(w))\cap I^*=\emptyset$ and $n(G)=0$, otherwise. Theorem 2.6 in \cite{martin} proves that $\gr(G)=\alpha(G)+n(G)$ and  characterizes all \textit{Gds}'s of $G$, which can be obtained in polynomial time. 
 
We analyze the value of $\gr(G,I)$, for all independent set $I$ of $G$. 

First, for each independent set $I$ of $G$ we define a sequence $S(I)\in \mathcal L(G,I)$. Observe that, for all $I$ we have $|K\cap I|\leq 1|$. Then:
\begin{enumerate}
	\item If $I\cap K=\emptyset$, $S(I)= (I^*)$ if $n(G)=0$ and, otherwise, $S(I)= (N(u)\cap I^*)\oplus (u) \oplus (I^*\setminus N(u))$ for any $u\in K$ such that there exists $v\in K$ with $N(u)\cap N(v)\cap I^*=\emptyset$.
	\item If $I\cap K=\{u\}$, $S(I)=(u)\oplus (I\setminus N(u))$.
\end{enumerate}

Clearly, $S(I)\in \mathcal L(G,I)$ for all non empty independent set $I$ of $G$.

Now, we can prove:
\begin{lemma}
Let $G=(I^*\cup K, E)$ be a split graph and $I$, an independent set of $G$. Then, $\gr(G,I)=|S(I)|$. That is: 
\begin{enumerate}
	\item If $K\cap I=\emptyset$, $\gr(G,I)=|I^*|+n(G)$.
	\item If $K\cap I=\{u\}$, $\gr(G,I)= |I^*\setminus N(u)|+1$.
\end{enumerate}
\end{lemma}

\begin{proof}
We only need to prove that, for any $S\in \mathcal L(G,I)$, $|S|\leq |S(I)|$.

Let $S\in \mathcal L(G,I)$. Clearly, if $\hat S\cap K=\emptyset$, $K\cap I=\emptyset$ holds and $|S|\leq |I^*|\leq |I^*|+n(G)$. 
Then, assume that $\hat S\cap K\neq \emptyset$. Let $u$ be the vertex in $K$ of minimum order in $\hat S$. Observe that, for all $v\in \hat S\cap K$, $v\neq u$, there exists $i_v\in PN_S(v)\cap I^*$. Clearly, $i_v\notin \hat S$. Then, $S_{v\hookleftarrow i_v}\in  \mathcal L (G, I\cap\{i_v\})$. Then, we can assume that $S\in \mathcal L(G,I)$ such that $\hat S\cap K=\{u\}$.
\begin{enumerate}
	\item Let $I$ such that $K\cap I=\emptyset$.
If there exists $w\in PN_S(u)\cap I^*$, $w\notin \hat S$. Then, $|S|=|\hat S\cap I^*|+1\leq (|I^*|-1)+1= |I^*|\leq |I^*|+n(G)$.

Now, assume that $PN_S(u)\cap I^*=\emptyset$. Then, for all $w\in N(u)\cap I^*$ , $w\in \hat S$ and $O_S(w)<O_S(u)$. Observe that if $n(G)=0$, $\bigcup\limits_{w\in N(u)\cap I^*}N(w)=K$ and $PN_S(u)=\emptyset$, a contradiction. Then, $n(G)=1$ and we have 
$|S|=|\hat S\cap I^*|+1\leq |I^*|+1= |I^*|+n(G)$.
\item
Let $I$ such that $K\cap I=\{u\}$. Since $u\in I_S=I$, for all $w\in I^*\cap N(u)$, $w\notin \hat S$. Then, $|S|\leq |I^*\setminus N(u)|+1$.
\end{enumerate}
\end{proof}

Then, by Theorem \ref{gdsI}, we have: 

\begin{theorem} \label{pseudo-splitS}
Let $G=(I^*\cup K, E)$ be a split graph and $\R=\left\{G_v:v\in I^*\cup K\right\}$. Then
$$\gr(\GR)= \max \left\{\sum_{v\in I^*} \gr(G_v)+n(G), \;\; \max_{v\in K}\left\{\gr(G_v)+\sum_{w\in I^*\setminus N(v)} \gr(G_w)\right\}\right\}.$$
Moreover, given $S_v\in \mathcal Gr(G_v)$ for all $v\in I^*\cup K$, a \textit{Gds} of $\GR$ can be obtained in polynomial time.
\end{theorem}

As a direct consequence of previous theorem, a formula for the lexicographic product of a split graph and a graph $H$ can be obtained.
 
\begin{corollary} \label{lexisplits}
Let $G=(I^*\cup K, E)$ be a split graph and $H$ be a graph. 
Then, 
$\gr(G\circ H)= |I^*|\gr(H)+n(G)$.
\end{corollary}

\section{\textit{Gds}'s on graphs with few $P_4$'s}\label{SectionFew}

In this section we study \textit{Gds}'s in the following three \emph{few} $P_4$'\emph{s} graph classes. Let $U$ be a subset of vertices inducing a $P_4$ in $G$. A \emph{partner} of $U$ is a vertex $v\in G\setminus U$ such that $U\cup\{v\}$ induces at least two $P_4$ in $G$. A graph is called  \emph{partner limited graph} ($PL$, for short) if any $P_4$ in $G$ has at most two partners \cite{Roussel}. In addition, a graph is \emph{extended $P_4$-laden} ($EP_4L$, for short) if every induced subgraph with at most six vertices contains at most two induced $P_4$'s or it is $\{2K_2, C_4\}$-free \cite{laden1}.
Finally, a graph $G$ is a $(q,t)$-graph if every set of at most $q$ vertices induces at most $t$ distinct $P_4$'s \cite{qq-4}. In particular, for a fixed $q$ and $t=q-4$ we obtain the class of $(q,q-4)$-graphs. Observe that, when $q=4$ we have the class of cographs.

These classes are on the top of a widely studied hierarchy of many known graph classes containing few $P_4$'s, including cographs, $P_4$-sparse, $P_4$-lite, $P_4$-laden and $P_4$-tidy graphs (see Figure \ref{pocosp4}). Besides, their prime factors with respect to the $X$-join product are completely characterized. These facts drive to the study of combinatorial problems on these graph classes.

\medskip

\begin{figure}[h]
	\centering
		\includegraphics[scale=0.13]{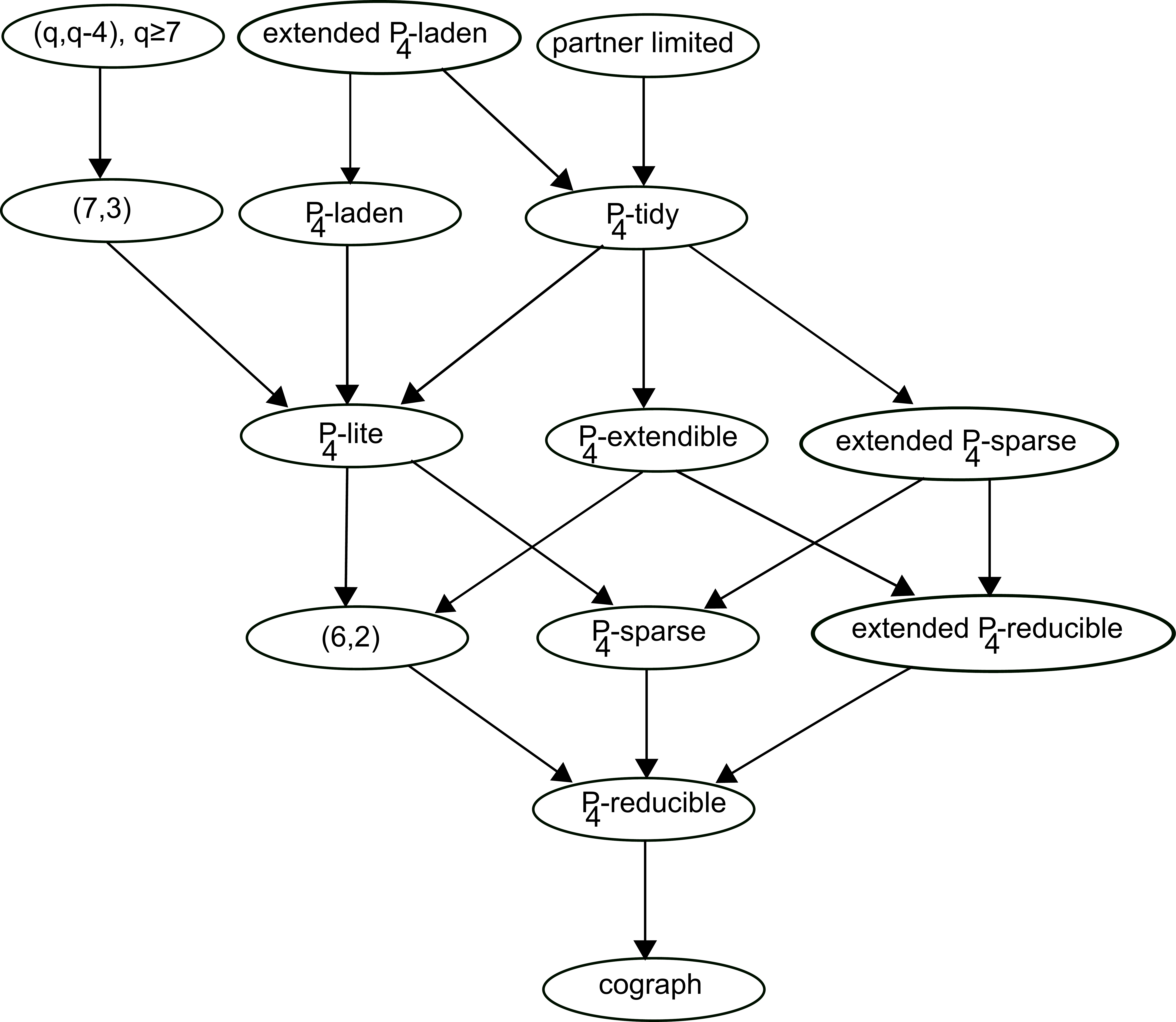}
	\caption{Graphs with few $P_4$'s}
	\label{pocosp4}
\end{figure}

As we have mentioned in the introduction, the modular decomposition of a graph $G$ obtains, in linear time, the family of its modular subgraphs and the sequence of disjoint union and complete joint operations needed in order to reconstruct $G$ from its modular subgraphs. Then, for each modular subgraph $H$ of $G$, if $H$ is not a prime graph, the procedure gives, also in linear time, the main factor $H'$ and the family $\R=\{H_v: v\in V(H)\}$ such that $H=H'\hookleftarrow \R$.

In particular, 
the prime graphs in $PL$ 
are paths and their complements, cycles and their complements, a family of graphs with at most $9$ vertices, and a subclass of split graphs \cite{Roussel}. For $EP_4L$, the prime graphs 
are $P_5$, $\overline{P_5}$ and $C_5$ and split graphs \cite{laden1}.
Besides, for $\mathcal F\in \left\{PL, EP_4L\right\}$, if $H\in\mathcal{M(F)}$ and it is not prime, then $H$ is isomorphic to a graph $\GR$ where the main factor $G$ is a split graph and every graph in $\mathcal R$ belongs to $\mathcal F$. Note that the Grundy domination number of the complement of paths and the complement of cycles with at least five vertices is $3$ and a \textit{Gds} can be easily obtained. Then, from Theorem \ref{pseudo-splitS}, we can obtain the Grundy domination number and a \textit{Gds} of any graph in $\mathcal M(\mathcal F)$ in polynomial time.

Finally, Remark \ref{union} implies the following:

\begin{theorem}\label{PLEP4L}
A \textit{Gds} can be obtained in polynomial time for $PL$ and $EP_4L$ graphs.
\end{theorem}

In addition, for any fixed $q\geq 4$, the 
prime graphs in $(q,q-4)$-graphs 
are \emph{prime spider} graphs or some graphs of size at most $q$ 
(see \cite{CMR2000}). It is not hard to see that prime spider graphs are split graphs $G$ for which $\eta(G)$ is linear time computable and then, we can obtain a \textit{Gds} for prime spider graphs in linear time.
Besides, if $H\in\mathcal{M}(q,q-4)$ and it is not prime, then $H$ is isomorphic to a graph $\GR$ where the main factor $G$ is a prime spider or a graph with at most $q$ vertices and the graphs in $\mathcal R$ are trivial graphs except exactly one graph which is in $(q,q-4)$ \cite{qq-4}.
Finally, note that if $q$ is fixed, by Lemma \ref{remplazovertice}, $\gr(G_{u\hookleftarrow H})$ can be obtained in linear time for any graph $G$ with order at most $q$, $u\in V(G)$, and a graph $H$ such that $\gr(H)$ is linear time computable. 

These facts together with Theorem \ref{pseudo-splitS} and Remark \ref{union} imply the following result.

\begin{theorem}\label{Tqq-4}
A \textit{Gds} can be computed in linear time for any $(q,q-4)$-graph.
\end{theorem}

\section{Concluding remarks}

The here presented results generalize the results on \textit{Gds}'s and the lexicographic product of graphs presented in \cite{Bresar1} in two directions. First, we extend the results on this product to a more general one, the $X$-join product of graphs. In second place, we enlarge the family of graphs where closed formulas for Grundy domination number are known. 
Moreover, we also give similar results for split graphs.

We consider the $X$-join product of a power of cycles or power of paths with a family of graphs with given Grundy domination numbers, and show that the Grundy domination number of this product can be obtained in polynomial time based on a polynomial reduction to the $MWIS$ problem.


In a similar way, the results in this paper include generalizations of results on \textit{Gds}'s for split graphs and cographs, given in 
\cite{martin}. Indeed, since $EP_4L$ graphs is a superclass of split graphs, Theorem \ref{PLEP4L} provides a superclass of these graphs where obatining a \textit{Gds} is polynomial-time solvable and Theorem \ref{Tqq-4} gives a superclass of cographs where the problem in linear time solvable.

Additionally, the superclass of $EP_4L$ called \emph{fat-extended $P_4$ laden} was introduced in \cite{ALS2012} from the modular decomposition of $EP_4L$ considering that the graphs $P_5$, $\overline{P_5}$ and $C_5$ are not only prime graphs but also main factors of modular graphs in the graph class. Then, note that the reasoning applied in Section \ref{SectionFew} infers that a \textit{Gds} can be computed in polynomial time for any fat-extended $P_4$ laden graph.



\begin{thebibliography}{00}
\bibitem{Zero} AIM Minimum Rank-Special Graphs Work Group, \emph{Zero-forcing sets and the minimum rank of graphs}, Linear Algebra Appl. 428 (2008) 1628--1648.
\bibitem{ALS2012} J. Araujo, C. Linhares Sales, \emph{On the Grundy number of graphs with few $P_4$'s}, Discrete Applied Mathematics 160(18) (2012), 2514--2522.
\bibitem{FewP4s} L. Babel, T. Kloks, J. Kratochvíl, D. Kratsch, H. M\"{u}ller, S. Olariu, \emph{Efficient algorithms for graphs with few $P_4$'s}, Discrete Mathematics 235(1-3) (2001), 29--51.
\bibitem{qq-4} L. Babel, S. Olariu, \emph{On the structure of graphs with few} $P_4$'\emph{s}, Discrete Applied Mathematics 84 (1998) 1--13.
\bibitem{Bresar1} B. Bre\v{s}ar, C. Bujt\'as, T. Gologranc, S. Klav\v{z}ar, G. Ko\v{s}mrlj, B. Patk\'os, Z. Tuza, M. Vizer, \emph{Dominating Sequences in Grid-Like and Toroidal Graphs}, The Electronic Journal of Combinatorics 23(4) (2016) \#P4.34
\bibitem{BresarZero} B. Bre\v{s}ar, C. Bujt\'as, T. Gologranc, S. Klav\v{z}ar, G. Ko\v{s}mrlj, B. Patk\'os, Z. Tuza, M. Vizer, \emph{Grundy dominating sequences and zero forcing sets}, Discrete Optim. 26 (2017) 66--77.
\bibitem{martin} B. Bre\v{s}ar, T. Gologranc, M. Milani\v{c}, D. F. Rall, R. Rizzi, \emph{Dominating sequences in graphs}, Discrete Mathematics, 336 (2014) 22--36.
\bibitem{CMR2000} B. Courcelle, J. A. Makowsky and U. Rotics, \emph{Linear Time Solvable Optimization Problems on Graphs of Bounded Clique Width}, Theory of Computing Systems \textbf{33}  (2000), 125--150.
\bibitem{mwis} S. Felsner, R. Muller, L. Wernisch, \emph{Trapezoid graphs and generalizations, geometry and algorithms},
Discrete Appl. Math. 74 (1997) 13--32.

\bibitem{cordales} A. Frank, \emph{Some polynomial algorithms for certain graphs and hypergraphs}.
Proc. 5th Br. comb. Conf., Aberdeen 1975, Congr. Numer. XV (1976) 211--226 .

\bibitem{laden1} V. Giakoumakis, \emph{$P_4$-laden graphs: a new class of brittle graphs}, Information Processing Letters, 60 (1996), 29--36.
\bibitem{book}  Richard Hammack, Wilfried Imrich, Sandi Klav\v {a}r, \emph{Handbook of product graphs}, Second edition, CRC Press 2011.


\bibitem{Hir51} T. Hiragushi, \emph{On the dimension of partially ordered sets}, Sci. Rep. Kanazawa University (1951) 77--94.

\bibitem{Jol73} J.L. Jolivet, \emph{Sur le joint d'une famille des graphes}, Discrete Mathematics 5 (1973) 145--158.

\bibitem{modular1} R. M. McConnel, J. Spinrad, \emph{Linear-time modular decomposition and efficient transitive orientation of comparability graphs}, Proceedings of the fifth annual ACM-SIAM symposium on Discrete Algorithms, SODA (1994) 536--545.

\bibitem{clawfree} G.J. Minty, \emph{On maximal independent sets of vertices in claw-free graphs}
J. Combin. Theory Ser. B, 28 (1980), pp. 284-304

\bibitem{MR84} R.H. M\"{o}hring, F.J. Radermacher, \emph{Substitution decomposition and connection with combinatorial optimization}, Ann. Discrete Mathematics 19 (1984) 257--356.

\bibitem{Roussel} F. Roussel, I. Rusu, H. Thuillier, \emph{On graphs with limited number of $P_4$-partners}, Internat. J. Found. Comput. Sci. 10 (1999) 103--121.

\bibitem{Sabidussi} Sabidussi, \emph{Graph derivatives}, G. Math Z (1961) 76: 385.


\end{thebibliography}
\end{document}